\documentclass[11pt,a4,fleqn]{article}
\usepackage{graphicx}
\usepackage{amsmath,amssymb,latexsym,graphics,epsfig}
\usepackage{hyperref}
\usepackage{color}
\usepackage{amsthm}

\newtheorem{theorem}{\bf Theorem}[section]

\newtheorem{lemma}[theorem]{\bf Lemma}
\newtheorem{corollary}[theorem]{\bf Corollary}
\newtheorem{conjecture}[theorem]{\bf Conjecture}

\newtheorem{remark}[theorem]{\bf Remark}

\numberwithin{equation}{section}

\begin{document}
\title{{\Large Size Ramsey Numbers of Stars versus Cliques}}

\author{M. Miralaei$^{\textrm{a}}$, G.R. Omidi$^{\textrm{a},\textrm{b},1}$, M. Shahsiah$^{\textrm{b}}$ \\[2pt]
{\small $^{\textrm{a}}$Dept. Mathematical Sciences, Isfahan University of Technology},\\
{\small Isfahan, 84156-83111, Iran}\\
{\small $^{\textrm{b}}$School of Mathematics, Institute for Research in Fundamental Sciences (IPM),}\\
{\small P.O. Box 19395-5746, Tehran, Iran }\\[2pt]
{m.miralaei@math.iut.ac.ir, romidi@cc.iut.ac.ir,
shahsiah@ipm.ir}}

\date{}

\maketitle \footnotetext[1] {This research is partially
carried out in the IPM-Isfahan Branch and in part supported
by a grant from IPM (No. ????????).} \vspace*{-0.5cm}

\begin{abstract}
The size Ramsey number $ \hat{r}(G,H) $ of two graphs $ G $ and $ H $ is the smallest integer $ m $ such that there exists a graph $ F $ on $ m $ edges with the property that every red-blue colouring of the edges of $ F $, yields a red copy of $ G $ or a blue copy of $ H $.
In $ 1981 $, Erd\H{o}s observed that
$\hat{r}(K_{1,k},K_{3})\leq \binom{2k+1}{2}-\binom{k}{2}$ and he conjectured that the corresponding upper bound on $ \hat{r}(K_{1,k},K_{3}) $ is sharp.
In $ 1983 $, Faudree and Sheehan extended this conjecture as follows:
\begin{eqnarray*}
\hat{r}(K_{1,k},K_{n})=\left \{
\begin{array}{lr}
\binom{k(n-1)+1}{2}-\binom{k}{2} & ~k\geq n~ \text{or}~ k~ \text{odd}.\\
\\
\binom{k(n-1)+1}{2}-k(n-1)/2 & \text{otherwise}.
\end{array} \right.
\end{eqnarray*}
They proved the case $ k=2 $. In $ 2001 $, Pikhurko showed that this conjecture is not true for $ n=3 $ and
$ k\geq 5 $, disproving the mentioned conjecture of Erd\H{o}s. Here we prove Faudree and Sheehan's conjecture for a given $ k\geq 2 $
and $ n\geq k^{3}+2k^{2}+2k  $.\\


\end{abstract}

\section{Introduction}

Given two graphs $ G $ and $ H $, we say that $ F\longrightarrow (G,H) $,
if for any red-blue colouring of the edges of $ F $ we have a red copy of $ G $ or a blue copy of $ H $.
The {\it size Ramsey number} $ \hat{r}(G,H) $ of two graphs $ G $ and $H$ is the minimum number of edges of a graph $ F $ such that $ F\longrightarrow (G,H) $. Using this notation, the {\it Ramsey number} $r(G,H)$ is the minimum integer
$n$ such that $K_{n}\longrightarrow (G,H)$.
We also define the {\it restricted size Ramsey number} $\hat{r}^{*}(G,H)$ for two graphs $G$ and $H$ as follows:
\begin{eqnarray*}
\hat{r}^{*}(G,H)=min\{\vert E(F) \vert : F\longrightarrow (G,H), \vert V(F) \vert =r(G,H)\}.
\end{eqnarray*}
Clearly for every two graphs $G$ and $H$, we have $ \hat{r}(G,H) \leq \hat{r}^{*}(G,H) $.
Also, by the definition of $r(G,H)$, we have
$K_{r(G,H)}\longrightarrow (G,H)$.
Since the complete graph on $ r(G,H) $
vertices has
$ r(G,H) \choose 2 $
edges, we obtain trivially
\begin{eqnarray}\label{1}
\hat{r}(G,H)\leq {r(G,H)\choose 2}.
\end{eqnarray}
Chv\'{a}tal showed that equality holds in (\ref{1}), when $ G $ and $ H $ are complete graphs (see \cite{3}).\\

The investigation of the size Ramsey numbers of graphs was initiated by
Erd\H{o}s et al. \cite{3} in $1978$.
In this paper, we investigate the size Ramsey number of $K_{1,k}$, the star with $k$ edges,
versus the complete graph $K_{n}$. These numbers were first considered by
Erd\H{o}s et al. \cite{3}. They showed the following asymptotic result:
\begin{theorem} \label{thm1}{\rm \cite{3}}
Let $\varepsilon$ be a fixed real number satisfying $ 0<\varepsilon<1  $ and let $n\geq 3$ be a fixed natural number. If $k$ is sufficiently large, then
\begin{eqnarray*}
\hat{r}(K_{1,k},K_{n})\geq \max \{k^{2}/2, (1-\varepsilon) \lfloor (n-2)^{2}/4\rfloor k^{2}/2\}.
\end{eqnarray*}
\end{theorem}
\noindent Let the graph $ K_{k+1}+\overline{K}_{k} $ be obtained from $ K_{k+1} $ by considering $ k $ new vertices and joining each vertex of $ K_{k+1} $ to all these $ k $ additional vertices.
In \cite{2}, Erd\H{o}s observed that
$K_{k+1}+\overline{K}_{k}\longrightarrow (K_{1,k},K_{3})$ and conjectured that the corresponding upper bound on $ \hat{r}(K_{1,k},K_{3}) $ is sharp.
Faudree and Sheehan generalized this result and showed that:
\begin{theorem}\label{thm2} {\rm \cite{4}}
Let $k,n\geq 2 $, then
\begin{eqnarray*}
\hat{r}^{*}(K_{1,k},K_{n})=\left \{
\begin{array}{lr}
\binom{k(n-1)+1}{2}-\binom{k}{2} & ~k\geq n~\text{or}~ k~ \text{odd},\\
\\
\binom{k(n-1)+1}{2}-k(n-1)/2 & ~\text{otherwise}.
\end{array} \right.
\end{eqnarray*}
\end{theorem}
\noindent They also posed the following conjecture, generalizing the mentioned conjecture of Erd\H{o}s on 
$ \hat{r}(K_{1,k},K_{3}). $
\begin{conjecture}\label{conj1} {\rm \cite{4}}
Let $k,n \geq 2$. Then $\hat{r}(K_{1,k},K_{n})=\hat{r}^{*}(K_{1,k},K_{n})$.
\end{conjecture}
\noindent
They proved the case $k=2$ of this conjecture (see \cite{4}). Pikhurko \cite{7}, with a nice counterexample, disproved the
Erd\H{o}s conjecture on $ \hat{r}(K_{1,k},K_{3}) $ for $ k \geq 5 $ (the case $ n=3 $ of Conjecture \ref{conj1}).
More precisely, he showed that $ \hat{r}(K_{1,k},K_{3})<k^{2}+\sqrt{2}k^{3/2}+k $ for $ k\geq 1 $. One can easily check that  for $ k\geq 5 $, we have
\begin{eqnarray*}
k^{2}+\sqrt{2}k^{3/2}+k < \binom{2k+1}{2}-\binom{k}{2}.
\end{eqnarray*}
Also, Pikhurko \cite{8} showed that for any graph $ F $ with chromatic number $ \chi(F)\geq 4 $,
\begin{eqnarray*}
\hat{r}(K_{1,k}, F)\leq \chi(F)\big(\chi(F)-2\big)k^{2}/2 +o(k^{2})
\end{eqnarray*}
and he conjectured that this is sharp. He proved his conjecture for the case $ \chi(F)=4 $.\\

In this paper, we show that for a fixed $k\geq 2$, Conjecture \ref{conj1}
holds for\linebreak $n\geq k^{3}+2k^{2}+2k$. More precisely, we demonstrate the following theorem.
\begin{theorem} \label{thm3}
Let $ k\geq 2 $ and $n\geq k^{3}+2k^{2}+2k$. Then
\begin{eqnarray*}
\hat{r}(K_{1,k},K_{n})=\hat{r}^{*}(K_{1,k},K_{n})= \left \{
\begin{array}{lr}
\binom{k(n-1)+1}{2}-\binom{k}{2} & ~\text{if}~ k ~\text{is odd},\\
\\
\binom{k(n-1)+1}{2}-k(n-1)/2 & \text{if} ~k ~\text{is even}.
\end{array} \right.
\end{eqnarray*}
\end{theorem}
\noindent Note that, we also make no attempt to give out a better lower bound for $ n $ in terms on $ k $ in Theorem \ref{thm3}. Throughout the paper, for the sake of clarity of presentation, we omit floor and ceiling signs whenever they are not crucial.

\noindent
\textbf{Conventions and Notations:}
For a graph $ G $, we write $ V(G) $,  $ E(G) $ and $ e(G) $ for the vertex set, edge set and the number of edges of $ G $, respectively.
For $ v \in V(G) $, by $ N_{G}(v) $ we mean the set of all neighbors of $ v $ and the degree $ d_{G}(v) $ of $ v $ is
$ d_{G}(v)=\vert N_{G}(v) \vert $. We denote by
$ \delta(G) $ and $ \Delta(G) $ the minimum and maximum degrees of $ G $, respectively.
Let $ X\subseteq V(G) $. Then $ G[X] $ is the induced subgraph of $ G $ with vertex set $ X $. We write $ G\setminus X $ for $ G[V(G)\setminus X] $.
Let $ A,B \subset V(G) $, then $ e(A,B)=\vert \{ \{x,y\}\in E(G) : x\in A, y\in B\}\vert $,
is the number of edges connecting a vertex of $ A $ to a vertex of $ B $.
 By $ \overline{G} $ we mean the complement of $ G $.\\
\section{Preliminaries}
In this section, we prove some results that will be used in the follow up section. We also recall some results from \cite{4} and \cite{8}.
The following theorem is indeed a special case of Theorem $ 3.1$ of \cite{8}.
\begin{theorem}\label{them} {\rm \cite{8}}
Let $ k\geq 2 $ and $ n\geq 1 $. If $ G $ is a graph so that $ G\longrightarrow(K_{1,k},K_{n}) $, then $ e(G)\geq k^{2}\binom{n-1}{2} $.
\end{theorem}
\noindent we also use the following lemma of Pikhurko [$ 6 $, Lemma $ 5.1 $].
\begin{lemma}\label{lem 2} {\rm \cite{8}}
Let $ G $ be a graph so that $ G\longrightarrow(K_{1,k}, K_{n}) $. For any set $ S\subset V(G) $, there exists
$ T\subset V(G)$ such that $ \Delta(G\setminus T)<k $, each vertex of $ T$ sends at least $ k $ edges to
$ V(G)\setminus T$ and $ T $ is incident to at least $ k(\vert T \vert - \vert S \vert)+e(S,V(G)) $ edges.
\end{lemma}
The following lemma is a modified version of [$ 4 $, Lemma $ 1 $]. But, for the sake of completeness, we state a proof here.
\begin{lemma}\label{lem1} {\rm \cite{4}}
Let $ k \geq 2 $ and G be a graph with $ e(G) \geq \binom{k}{2}+1 $. Then either \vspace{0.2 cm}\\
$ (i) $ G contains an induced subgraph with $ k+1 $ vertices and minimum degree at least\\
\indent $ 1 $ or\\
$ (ii) $ G contains a matching $ M $ with $ \vert M \vert=e(G)$.
\end{lemma}
\begin{proof}
We use induction on $ k $. The case $ k=2 $ is easily verified. Suppose that $ k\geq 3 $. If there is a vertex $ v\in V(G) $ so that
$ d_{G}(v)\geq k $, then the induced subgraph on $ A\cup \{v\} $ has $ k+1 $ vertices with minimum degree at least one, where $ A\subset N_{G}(v) $ is a set containing $ k $ vertices.
Hence we may assume that $ d_{G}(v)\leq k-1 $, for all $ v\in V(G) $.
Furthermore, there exists $ v\in V(G) $ such that $ d_{G}(v)\geq 2 $, otherwise Lemma \ref{lem1} $ (ii) $ holds. Now, choose $ v\in V(G) $ so that
$ 2\leq d_{G}(v) \leq k-1$. Set $ G'=G\setminus \{v\} $. Clearly $ e(G')\geq e(G)-(k-1)\geq \binom{k-1}{2}+1 $. So by the induction hypothesis $ V(G') $ contains a subset $ Y $ so that, either\\
$ (a) $ $ \vert Y \vert=k $ and $ \delta(G'[Y])\geq 1 $, or\\
$ (b) $ $ G'[Y]\cong sK_2 $, so that $ s=e(G') $.\\
At first assume that $ N_{G}(v) \cap Y\neq \emptyset $. If $ Y $ is of type $ (a) $, then set $ X=Y\cup \{v\} $ and Lemma \ref{lem1}
$ (i) $ holds. So assume that $ Y $ is of type $ (b) $. If $ k $ is odd, then the set of $ k+1 $ vertices incident to some set of
$ (k+1)/2 $ disjoint edges contained in
$ G'[Y] $ satisfies Lemma \ref{lem1} $ (i) $. Now suppose that $ k $ is even and $ u\in N_{G}(v)\cap Y $. Let $ X $ be the set of
$ k $ vertices incident to $ k/2 $ disjoint edges in $ G'[Y] $ including an edge incident to $ u $. Clearly $ X\cup \{v\} $ satisfies Lemma \ref{lem1} $ (i) $.\\
Now assume that $ N_{G}(v)\cap Y=\emptyset $. Let $ v_1, v_2 \in N_{G}(v)$. First, suppose that $ Y $ is of type $ (b) $. If $ k $ is odd, then by an argument similar to the previous paragraph we can find a subgraph in $ G $ with $ k+1 $ vertices and minimum degree at least $ 1 $. If $ k $ is even, then let $ X^* $ be the set of vertices incident to some subset of $ (k-2)/2 $ disjoint edges in $G'[Y]$. Set $ U=X^*\cup \{v, v_1, v_2\} $. Clearly $ G[U] $ satisfies Lemma \ref{lem1} $ (i) $.
So we may assume that $ Y $ is of type $ (a) $. Choose any vertex $ u' \in Y $ and write
$ X=(Y\setminus \{u'\}) \cup \{v, v_1\} $. If $ X $ does not satisfy Lemma \ref{lem1} $ (i) $, then there exists $ x\in X $ so that
$ d_{G[X]}(x)=0 $.
Since $ Y $ is of type $ (a) $, so $ x\in Y\setminus \{u'\} $ and $ x\sim u' $ in $ G $. So we have proved that any vertex $ u'\in Y $ is incident to some vertex $ x\in Y $ so that $ d_{G'[Y]}(x)=1 $. In particular this is true for each vertex $ x $ with $ d_{G'[Y]}(x)=1 $. Hence, $ G'[Y] $ is a matching and $ k $ is even.
Now set $ U=X^* \cup \{v,v_1,v_2\} $ where $ X^* $ is the set of $ k-2 $ vertices incident to $ (k-2)/2 $ disjoint edges in
$ G'[Y] $. Clearly $ G[U] $ satisfies Lemma \ref{lem1} $ (i) $. So we are done.
\end{proof}

The following result is an immediate consequence of Lemma \ref{lem1}.
\begin{corollary}\label{cor1} {\rm[$ 4 $, Lemma $ 2 $]}
Let $ k\geq 3 $ and $ G $ be a graph with $ e(G) \geq \binom{k}{2}+1 $. If $ k $ is odd, then $G$ contains an induced subgraph with $ k+1 $ vertices and minimum degree at least $ 1 $.
\end{corollary}
\begin{remark}\label{rm1}
Let $ G $ be a graph so that $ G\longrightarrow (K_{1,k},K_{n}) $. If $ G $ is edge minimal, then each vertex of $ G $ must be in some clique $ K_{n} $. Otherwise, assume that some vertex $ v\in V(G) $ is not in any clique $ K_n $. Colour the edges of $G'=G\setminus \{v\} $ red or blue arbitrarily and extend this colouring to $ G $ by colouring the edges incident to $ v $ blue. Since $ v $ is not in any blue copy of $ K_n $,
so $ G' $ has either a red copy of $ K_{1,k} $ or a blue copy of $ K_{n} $ and so $ G'\longrightarrow (K_{1,k},K_n) $, which is a contradiction with the edge minimality of $ G $.
\end{remark}
\begin{lemma}\label{lemmm}
Let $ k\geq 2 $ and $ n\geq 3k+3 $. Let $ H $ be a graph
with $ R+t $ vertices, where $ R=r(K_{1,k},K_{n})=k(n-1)+1 $ and
$ 0\leq t \leq \lfloor \dfrac{kn-2k^{2}}{(k+1)^{2}}\rfloor $. Set
\begin{eqnarray*}
R'= \left \{
\begin{array}{lr}
\binom{k}{2}+1 & ~\text{if}~ k ~\text{is odd},\\
\\
k(n-1)/2+1 & \text{if} ~k ~\text{is even}.
\end{array} \right.
\end{eqnarray*}
If $ e(H)\geq Rt+\binom{t}{2}+R' $, then $ H $ contains $ t+1 $ disjoint subsets
$ A_{1},\dots, A_{t+1} $ of vertices
so that for
$ 1\leq i \leq t+1 $, we have $ \vert A_{i} \vert = k+1$ and $ \delta(H[A_{i}])\geq 1 $.
\end{lemma}
\begin{proof}
We use induction on $ t $. First let $ t=0 $.
If there is no subset $ A_1 \subseteq V(H) $ with $ \vert A_1 \vert=k+1$ and $ \delta(H[A_1])\geq 1 $, then using Lemma \ref{lem1} and Corollary \ref{cor1}
we may assume that $ k $ is even
and $ H $ contains a matching $ M $ with $ \vert M \vert= e(H) $. But it is impossible, since
$ \vert V(H) \vert =k(n-1)+1$ and $ e(H)\geq k(n-1)/2+1 $. Now, let $ t\geq 1 $. Set
\begin{eqnarray*}
A=\{v \in V(H) : d_{H}(v)\geq (k+1)(t+1)\}.
\end{eqnarray*}
We have two following cases:\vspace*{0.5 cm}\\
\textbf{Case 1.}
$ A\neq \emptyset. $ \vspace*{0.5cm}\\
\noindent Set $ H'=H\setminus \{v\} $, where $ v \in A $. Since $ d_{H}(v)\leq R+t-1 $, we have
\begin{eqnarray*}
e(H') \geq R(t-1) +\binom{t-1}{2}+R'.
\end{eqnarray*}
By the induction hypothesis, $ H' $ contains $ t $ disjoint subsets
$ A_{1},\dots, A_{t} $ of vertices
so that for $ 1\leq i \leq t $, $ \vert A_{i} \vert = k+1$ and $ \delta(H'[A_{i}])\geq 1$. Choose
$ U\subseteq N_{H}(v) \setminus \bigcup_{i=1}^{t}A_{i}$ so that $ \vert U \vert =k$
(note that, this is possible since $ d_{H}(v) \geq (k+1)(t+1) $).
Clearly $ U \cup \{v\}$ is a new subset, disjoint from $ A_{i}s$, of order $ k+1 $ and minimum degree at least
$ 1 $ in $ H $. So we are done.\vspace*{0.5 cm}\\
\textbf{Case 2.}
$ A=\emptyset $.\vspace*{0.5cm}\\
\noindent By the induction hypothesis, $ H $
contains $ t $ disjoint subsets
$ A_{1},\dots, A_{t} $ of vertices
so that for $ 1\leq i \leq t $, $ \vert A_{i} \vert = k+1$ and $ \delta(H[A_{i}])\geq 1$.\\

Set
$ \mathcal{A}=\bigcup_{i=1}^{t}A_{i} $ and $ H'=H\setminus \mathcal{A} $.
Since each vertex of $ H $ has degree less than $ (k+1)(t+1) $, we have
\begin{eqnarray*}
e(H') \geq Rt+\binom{t}{2}+R'-t(k+1)\big((t+1)(k+1)-1\big)+t(k+1)/2.
\end{eqnarray*}
Since $ t\leq \lfloor\dfrac{kn-2k^{2}}{(k+1)^{2}}\rfloor $,
\begin{eqnarray*}
Rt+\binom{t}{2}+t(k+1)+t(k+1)/2-t(t+1)(k+1)^{2} \geq 0.
\end{eqnarray*}
So we have
$ e(H') \geq R' $.
If there exists a subset $ A_{t+1} $ of $ V(H') $ so that\linebreak
$ \vert A_{t+1}\vert=k+1 $ and
$ \delta(H'[A_{t+1}])\geq 1 $, then we are done.
Otherwise, using Lemma \ref{lem1} and Corollary \ref{cor1}, we may assume that $ k $ is even
and $ H' $ contains a matching $ M $ so that
$ \vert M \vert= e(H') \geq k(n-1)/2+1 $. But it is impossible,
since
$ \vert V(H') \vert =k(n-t-1)+1$. This completes the proof.
\end{proof}
\noindent Note that in the previous lemma we set $ n\geq 3k+3 $ to guarantee $ \lfloor\dfrac{kn-2k^{2}}{(k+1)^{2}}\rfloor\geq 1 $ .\\
Let $ G $ be a graph so that $ G\longrightarrow(K_{1,k},K_{n}) $. In the following theorem we present a sufficient condition on $ G $ so that $ e(G)\geq \hat{r}^{*}(K_{1,k},K_{n}) $.
\begin{theorem}\label{thm a}
Let $ k\geq 2 $ and $ n\geq 3k+3 $. If $ G\longrightarrow (K_{1,k},K_n)$ and $ \vert G \vert=R+\ell $ so that $ R=r(K_{1,k}, K_n)=k(n-1)+1 $ and $ 0\leq \ell \leq \lfloor\dfrac{kn-2k^{2}}{(k+1)^{2}}\rfloor $, then $ e(G)\geq \hat{r}^{*}(K_{1,k}, K_{n}). $
\end{theorem}
\begin{proof}
Suppose to the contrary that $ e(G) < \hat{r}^{*}(K_{1,k},K_{n}) $.
So using Theorem \ref{thm2}, we have $ e(\overline{G}) \geq R\ell+\binom{\ell}{2}+R'$, where
\begin{eqnarray*}
R'= \left \{
\begin{array}{lr}
\binom{k}{2}+1 & ~\text{if}~ k ~\text{is odd},\\
\\
k(n-1)/2+1 & \text{if} ~k ~\text{is even}.
\end{array} \right.
\end{eqnarray*}
Using Lemma \ref{lemmm}, $ \overline{G} $ contains $\ell+1 $ disjoint subsets
$ A_{1},\dots, A_{\ell+1} $ of vertices
so that for
$ 1\leq i \leq \ell+1 $, we have $ \vert A_{i} \vert = k+1$ and $ \delta(\overline{G}[A_{i}])\geq 1 $.
Now consider the following colouring on $ G $. Partition $ V(G) $ into subsets
$ X_{1},\dots, X_{n-1} $ so that $ X_{i}=A_{i} $, for $ 1\leq i \leq \ell+1 $ and $ \vert X_{i} \vert =k$, for
$ i=\ell+2,\dots, n-1 $.
Colour every edge of $ G[X_{i}] $ red $ (1\leq i \leq n-1) $ and all other edges of $ G $ blue. Then there is no red copy of
$ K_{1,k} $ and no blue copy of $ K_{n} $. Which is a contradiction with the assumption that
$ G\longrightarrow (K_{1,k},K_{n}) $.\\
\end{proof}
\section{The proof of Theorem \ref{thm3}}
\begin{proof}
Since $\hat{r}(K_{1,k},K_{n}) \leq \hat{r}^{*}(K_{1,k},K_{n})$, we shall prove just the lower bound for the claimed size Ramsey number. Let $ k\geq 2 $ and $ n\geq k^{3}+2k^{2}+2k $. Also let $ G $ be a graph so that
$ G\longrightarrow (K_{1,k},K_{n}) $. Without loss of generality we may assume that $ G $ is edge minimal.
Let $ \vert G \vert =R+\ell = k(n-1)+1+\ell $, where $ \ell \geq 0 $. We will show that
$ e(G)\geq \hat{r}^{*}(K_{1,k},K_{n}) $. If $ \ell \leq \lfloor\dfrac{kn-2k^{2}}{(k+1)^{2}}\rfloor $, then using Theorem \ref{thm a}, we are done. So we may assume that $ \ell> \lfloor\dfrac{kn-2k^{2}}{(k+1)^{2}}\rfloor $.\\
Set $ f(k,n)=\lfloor\dfrac{kn-2k^{2}}{(k+1)^{2}}\rfloor $ and
for $ j=1,\dots,n-3 $ set $ m_{j}=\max\{0,f(k,n-j)\} $. Let $ T_{0}=V(G) $. Clearly
$ G[T_{0}]\longrightarrow (K_{1,k},K_{n}) $
and
$ \vert T_{0} \vert =k(n-1)+1+\ell_{0}$, where  $ \ell_{0}=\ell $.
Repeat the following process as long as possible.\vspace*{0.5cm}\\
\textbf{Step 1}\\
Using Lemma \ref{lem 2}, for $ S=\emptyset $ there exists
$ T_{1}\subset T_{0} $ such that $ \Delta(G[T_{0}\setminus T_{1}])<k $ and each vertex $ x\in T_{1} $ sends at least $ k $ edges to $B_{1}=T_{0}\setminus T_{1}$. Let $ T_{1} $ be such a set with the minimum number of vertices. Note that $ G[T_{1}]\longrightarrow(K_{1,k},K_{n-1}) $.
To see this, colour the edges of $ G[T_{1}]$ arbitrarily and extend this colouring to $ G[T_{0}] $ by colouring the edges of $ B_{1} $ red and all so far uncoloured edges blue. Since $ G\longrightarrow (K_{1,k},K_{n}) $, then $ G[T_{1}] $ contains a red copy of $ K_{1,k} $
or a blue copy of $ K_{n-1} $. Therefore,
$ \vert T_{1} \vert\geq r(K_{1,k}, K_{n-1})=k(n-2)+1$. Let
$\vert T_{1} \vert=k(n-2)+1+\ell_{1}$, where $ \ell_{1}\geq 0 $. Clearly
$ \vert B_{1} \vert=\vert T_{0}\vert-\vert T_{1}\vert =k+\ell_{0}-\ell_{1}$.
Since each vertex of $ T_{1} $ sends at least $ k $ edges to $ B_{1} $, so $ \vert B_{1} \vert \geq k $ which implies that $ \ell_{1}\leq \ell_{0} $.
If $ \ell_{1}\leq m_{1} $, then stop. Otherwise go to Step $ 2 $.\vspace*{0.5cm}\\
\textbf{Step i} $ (2 \leq i \leq n-3) $\\
Since $ G[T_{i-1}]\longrightarrow (K_{1,k},K_{n-i+1}) $ by Lemma \ref{lem 2}, for $ S=\emptyset $ there exists
$ T_{i}\subset T_{i-1} $ such that $ \Delta(G[T_{i-1}\setminus T_{i}])<k $ and each vertex $ x\in T_{i} $ sends at least $ k $ edges to $B_{i}=T_{i-1}\setminus T_{i}$. Let $ T_{i} $ be such a set with the minimum number of vertices. Note that $ G[T_{i}]\longrightarrow(K_{1,k},K_{n-i}) $.
To see this, colour the edges of $ G[T_{i}]$ arbitrarily and extend this colouring to $ G $ by colouring the edges of $ G[B_{1}],\dots, G[B_{i}] $ red and all so far uncoloured edges blue. Since $ G\longrightarrow (K_{1,k},K_{n}) $, then $ G[T_{i}] $ contains a red copy of $ K_{1,k} $
or a blue copy of $ K_{n-i} $. Therefore,
$ \vert T_{i} \vert\geq r(K_{1,k}, K_{n-i})=k(n-i-1)+1$. Let
$\vert T_{i} \vert=k(n-i-1)+1+\ell_{i}$, where $ \ell_{i}\geq 0 $. Clearly
$ \vert B_{i} \vert=\vert T_{i-1}\vert-\vert T_{i}\vert =k+\ell_{i-1}-\ell_{i}$.
Since each vertex of $ T_{i} $ sends at least $ k $ edges to $ B_{i} $, so $ \vert B_{i} \vert \geq k $ which implies that $ \ell_{i}\leq \ell_{i-1} $.
If either $ \ell_{i}\leq m_{i} $ or $ i=n-3 $, then stop. Otherwise go to Step $ i+1 $.\\

Now, assume that the above procedure terminates in step $ j $. We have one of the following cases.
\vspace*{0.5 cm}\\
\textbf{Case 1.} $ j=1 $.\vspace*{0.5cm}\\
\noindent As $ n\geq k^{3}+2k^{2}+2k $, we have $ m_{1}=f(k,n-1) $. Since $ \ell_{1}\leq f(k,n-1) $, using Theorem \ref{thm a}, we have
$ e(G[T_{1}])\geq \hat{r}^{*}(K_{1,k},K_{n-1}) $. We have two following subcases.\vspace*{0.5cm}\\
\textbf{Subcase 1.1} $  \ell_{1}\geq \lceil k/2 \rceil $.\vspace*{0.5cm}\\
\noindent In this case $ \vert T_{1}\vert \geq k(n-2)+1+\lceil \frac{k}{2} \rceil$. Since each vertex $ x\in T_{1} $ sends at least
$ k $ edges to $ B_{1}$, we have
\begin{eqnarray*}
e(G)\geq e(G[T_{1}])+ k\vert T_{1}\vert \geq \hat{r}^{*}(K_{1,k},K_{n-1})+\dfrac{k^{2}(2n-3)+2k}{2}\geq \hat{r}^{*}(K_{1,k},K_{n}).
\end{eqnarray*}
So we are done.\vspace*{0.5cm}\\
\textbf{Subcase 1.2} $  \ell_{1}< \lceil k/2 \rceil $.\vspace*{0.5cm}\\
\noindent Clearly $ \vert B_{1}\vert =k+\ell_{0}-\ell_{1}\geq k+\ell_{0}-\lceil \frac{k}{2}\rceil +1= \ell_{0}+ \lfloor \frac{k}{2} \rfloor+1$.
Using Remark \ref{rm1} and the fact that $ \Delta(G[B_{1}])<k  $, we conclude that each vertex of $ B_{1} $ sends at least $ n-k $ edges to $ T_{1} $. Since $ n\geq k^{3}+2k^{2}+2k $, we have
\begin{align*}
e(G)\geq e(G[T_{1}])+\vert B_{1}\vert(n-k) &\geq \hat{r}^{*}(K_{1,k},K_{n-1})+(\ell _{0}+\lfloor \frac{k}{2} \rfloor +1)(n-k)\\
&\geq \hat{r}^{*}(K_{1,k},K_{n-1})+\big(\lfloor \dfrac{kn-2k^{2}}{(k+1)^{2}}\rfloor+\lfloor \frac{k}{2} \rfloor +2\big)(n-k)\\
&\geq \hat{r}^{*}(K_{1,k},K_{n-1})+\big(\dfrac{kn-2k^{2}}{(k+1)^{2}}+\lfloor \frac{k}{2} \rfloor +1\big)(n-k)\\
&\geq \hat{r}^{*}(K_{1,k},K_{n-1})+\dfrac{k^{2}(2n-3)+k}{2}\geq \hat{r}^{*}(K_{1,k},K_{n}).
\end{align*}
\vspace*{0.5 cm}\\
\textbf{Case 2.} $ 2\leq j \leq n-3k-3 $.\vspace*{0.5cm}\\
\noindent In this case we have $ m_{j}=f(k,n-j) $. Since $ \ell_{j}\leq m_{j} $, using Theorem \ref{thm a}, we have $ e(G[T_{j}])\geq \hat{r}^{*}(K_{1,k}, K_{n-j}) $. Therefore,
\begin{align*}
\hspace*{-1 cm}e(G) &\geq e(G[T_{j}]) +\vert T_{j} \vert kj+\sum_{i=2}^{j} \vert B_{i} \vert k(i-1)\\
&\geq \hat{r}^{*}(K_{1,k},K_{n-j})+\big(k(n-j-1)+1+\ell _{j}\big)kj+\sum_{i=2}^{j}\big(k+\ell _{i-1}-\ell _{i}\big)(i-1)k\\
&=\hat{r}^{*}(K_{1,k},K_{n-j})+k^{2}j(n-j-1)+kj+kj\ell _{j}+k^{2}\sum_{i=2}^{j}(i-1)
+k\sum_{i=2}^{j}(\ell_{i-1}-\ell _{i})(i-1)\\
&=\hat{r}^{*}(K_{1,k},K_{n-j})+k^{2}j(n-j-1)+kj+\frac{k^{2}j(j-1)}{2}+k\sum_{i=1}^{j}\ell _{i}.
\end{align*}
If $ k $ is even, using Theorem \ref{thm2}, we have $ \hat{r}^{*}(K_{1,k},K_{n-j})= k^{2}(n-j-1)^{2}/2 $. So
\begin{align*}
e(G) &\geq \dfrac{k^{2}(n-j-1)^{2}}{2}+k^{2}j(n-j-1)+kj+\dfrac{k^{2}j(j-1)}{2}+k\sum_{i=1}^{j}\ell_{i}\\
&=\dfrac{k^{2}(n-1)^{2}}{2}-\frac{k^{2}j}{2}+kj+k\sum_{i=1}^{j}\ell _{i}.
\end{align*}
Note that if $ k $ is even, then $\hat{r}^{*}(K_{1,k},K_{n})=k^{2}(n-1)^{2}/2$. So it suffices to show that
\begin{align*}
k\sum_{i=1}^{j}\ell _{i}+kj-\frac{k^{2}j}{2}\geq 0.
\end{align*}
Since for $ 1\leq i \leq j-1 $, we have $ \ell_{i}\geq \lfloor\dfrac{k(n-i)-2k^{2}}{(k+1)^{2}}\rfloor +1\geq\dfrac{k(n-i)-2k^{2}}{(k+1)^{2}} $, therefore
\begin{align*}
k\sum_{i=1}^{j}\ell _{i}+kj-\frac{k^{2}j}{2}&\geq k\sum_{i=1}^{j-1}\dfrac{k(n-i)-2k^{2}}{(k+1)^{2}}+kj-\frac{k^{2}j}{2}\\
&=k\big(\dfrac{kn(j-1)-k\frac{j(j-1)}{2}-2k^{2}(j-1)}{(k+1)^{2}}\big)-\frac{(k^{2}-2k)j}{2}\\
&=\dfrac{k(j-1)\big(2kn-kj-4k^{2}-(k-2)(k+1)^{2}\frac{j}{j-1}\big)}{2(k+1)^{2}}\geq 0.
\end{align*}
The last inequality is true since $ n\geq k^{3}+2k^{2}+2k $ and $ 2\leq j\leq n-3k-3 $ imply
\begin{eqnarray*}
2kn-kj-4k^{2}-(k-2)(k+1)^{2}\frac{j}{j-1}\geq 0.
\end{eqnarray*}
So we are done.\\ If $ k $ is odd, then $ \hat{r}^{*}(K_{1,k},K_{n})=\binom{k(n-1)+1}{2}-\binom{k}{2} $. To verify that $ e(G)\geq \hat{r}^{*}(K_{1,k},K_{n}) $, it suffices to show that
\begin{align*}
k\sum_{i=1}^{j}\ell _{i}+\frac{kj-k^{2}j}{2}\geq 0.
\end{align*}
The above inequality follows from a similar argument that used for the case $ k $ is even.\vspace*{0.5cm}\\
\textbf{Case 3.} $ n-3k-3<j\leq n-4 $.\vspace*{0.5cm}\\
In this case $ m_{j}=0 $. Note that $ G[T_{j}]\longrightarrow (K_{1,k},K_{n-j}) $ and $ \vert T_{j} \vert=r(K_{1,k},K_{n-j})+\ell_{j} $. Since $ 0\leq \ell_{j}\leq m_{j} $, we have
$ \vert T_{j} \vert=r(K_{1,k},K_{n-j}) $. Now, by the definition of $ \hat{r}^{*}(K_{1,k},K_{n-j})$, we have
$ e(G[T_{j}])\geq \hat{r}^{*}(K_{1,k},K_{n-j}) $.
By an argument similar to Case $ 2 $, we have
\begin{align*}
e(G) &\geq e(G[T_{j}]) +\vert T_{j} \vert kj+\sum_{i=2}^{j} \vert B_{i} \vert k(i-1)\\
&\geq \hat{r}^{*}(K_{1,k},K_{n-j})+k^{2}j(n-j-1)+kj+\frac{k^{2}j(j-1)}{2}+k\sum_{i=1}^{j-1}\ell _{i}.
\end{align*}
Our aim is to show that $ e(G)\geq \hat{r}^{*}(K_{1,k},K_{n}) $. If $ k $ is even, similar to Case $ 2 $, it suffices to show that
\begin{eqnarray*}
k\sum_{i=1}^{j-1}\ell _{i}+kj-\frac{k^{2}j}{2}\geq 0.
\end{eqnarray*}
Note that for $ 1\leq i\leq j-1,~ \ell_{i}> m_{i} $ and
\begin{eqnarray*}
m_{i}= \left \{
\begin{array}{lr}
f(k,n-i) & ~ 1 \leq i \leq n-3k-3,\\
\\
0 & ~ n-3k-2\leq i \leq j-1.
\end{array} \right.
\end{eqnarray*}
Since for $ 1\leq i \leq n-3k-3 $, we have $ \ell_{i}\geq \lfloor\dfrac{k(n-i)-2k^{2}}{(k+1)^{2}}\rfloor +1\geq\dfrac{k(n-i)-2k^{2}}{(k+1)^{2}} $, therefore
\begin{align*}
k\sum_{i=1}^{j-1}\ell _{i}+kj-\frac{k^{2}j}{2}&\geq k\sum_{i=1}^{n-3k-3}\dfrac{k(n-i)-2k^{2}}{(k+1)^{2}}+k\sum_{i=n-3k-2}^{j-1}1-\frac{(k^{2}-2k)j}{2}\\
&>k \sum_{i=1}^{n-3k-3}\dfrac{k(n-i)-2k^{2}}{(k+1)^{2}}-\frac{(k^{2}-2k)j}{2}\\
&= k\big(\dfrac{k(n-3k-3)\big(n-\frac{n-3k-2}{2}-2k\big)}{(k+1)^{2}}\big)-\frac{(k^{2}-2k)j}{2}\\
&=k\big(\dfrac{k(n-3k-3)(n-k+2)-(k+1)^{2}(k-2)j}{2(k+1)^{2}}\big)\geq 0.
\end{align*}
The last inequality is true, since $n\geq k^{3}+2k^{2}+2k $ and $ n-3k-2\leq j\leq n-4 $ imply
\begin{align*}
k(n-3k-3)(n-k+2)-(k+1)^{2}(k-2)j\geq 0.
\end{align*}
So when $ k $ is even, we are done.\\
Similarly, when $ k $ is odd, it can be shown that
$ e(G)\geq \hat{r}^{*}(K_{1,k},K_{n}) $.\vspace*{0.5cm}\\
\textbf{Case 4.} $j=n-3$.\vspace*{0.5cm}\\
\noindent In this case for every $ 1\leq i \leq n-4 $, we have
$  \ell_{i}> m_{i} $. Note that, $ G[T_{n-3}]\longrightarrow (K_{1,k},K_{3}) $ and\\
\begin{eqnarray*}
\vert T_{n-3}\vert =r(K_{1,k},K_{3})+\ell_ {n-3} =2k+1+\ell_{n-3}.
\end{eqnarray*}
Using Theorem \ref{them}, we have $ e(G[T_{n-3}])\geq k^{2} $. By an argument similar to Case
$ 2 $, we have
\begin{align*}
e(G) &\geq e(G[T_{n-3}]) +\vert T_{n-3} \vert k(n-3)+\sum_{i=2}^{n-3} \vert B_{i} \vert k(i-1)\\
&\geq k^{2}+(2k+1+\ell _{n-3})k(n-3)+\sum_{i=2}^{n-3}\big(k+\ell_ {i-1}-\ell _{i}\big)(i-1)k\\
&=k^{2}+2k^{2}(n-3)+k(n-3)+\frac{k^{2}(n-3)(n-4)}{2}+k\sum_{i=1}^{n-3}\ell _{i}.
\end{align*}
Again, we are going to show that $e(G) \geq \hat{r}^{*}(K_{1,k},K_{n}) $. When $ k $ is even, we have\linebreak
$ \hat{r}^{*}(K_{1,k},K_{n})=k^{2}(n-1)^{2}/2 $. It sufficies to show that
\begin{align*}
k\sum_{i=1}^{n-3}\ell _{i}+k^{2}+2k^{2}(n-3)+k(n-3)\geq \frac{k^{2}}{2}(5n-11).
\end{align*}
This inequality is certainly true if
\begin{align*}
2k\sum_{i=1}^{n-3}\ell _{i}+k^{2}+2kn\geq k^{2}n+6k.
\end{align*}
By an argument similar to Case $ 3 $, we have
\begin{align*}
2k\sum_{i=1}^{n-3}\ell _{i}+k^{2}+2kn&\geq 2k\sum_{i=1}^{n-3k-3}\dfrac{k(n-i)-2k^{2}}{(k+1)^{2}}+2k\sum_{i=n-3k-2}^{n-4}1+k^{2}+2kn\\
&=\dfrac{k^{2}(n-3k-3)(n-k+2)}{(k+1)^{2}}+2k(3k-1)+k^{2}+2kn\\
&=\dfrac{k^{2}(n-3k-3)(n-k+2)+(7k^{2}-2k+2kn)(k+1)^{2}}{(k+1)^{2}}\\
&\geq k^{2}n+6k.
\end{align*}
The last inequality holds, since $ n\geq k^{3}+2k^{2}+2k $.\\
If $ k $ is odd, then $ \hat{r}^{*}(K_{1,k},K_{n})=\binom{k(n-1)+1}{2}-\binom{k}{2} $. In this case, it suffices to show that
\begin{align*}
2k\sum_{i=1}^{n-3}\ell _{i}+2k^{2}+kn\geq k^{2}n+6k.
\end{align*}
Again, since $ n\geq k^{3}+2k^{2}+2k $, the above inequality holds. So we are done and the proof is completed.\\
\end{proof}

\end{document}